\documentclass[11pt,reqno]{amsart}
\usepackage{amsmath}
\usepackage{amssymb}
\usepackage[bookmarks]{hyperref}
\usepackage{enumerate}
\usepackage{color} 

\providecommand{\bysame}{\leavevmode\hbox to3em{\hrulefill}\thinspace}
\providecommand{\MR}{\relax\ifhmode\unskip\space\fi MR }

\providecommand{\href}[2]{#2}

\newtheorem{theorem}{Theorem}[section]
\newtheorem{lemma}[theorem]{Lemma}
\newtheorem{corollary}[theorem]{Corollary}
\newtheorem{proposition}[theorem]{Proposition}
\newtheorem*{WATconjecture}{WAT Conjecture}

\theoremstyle{definition}

\newtheorem{example}[theorem]{Example}

\numberwithin{equation}{section}

\newcommand{\B}{\mathbb{B}}
\newcommand{\C}{\mathbb{C}}
\newcommand{\D}{\mathbb{D}}
\renewcommand{\S}{\mathbb{S}}
\newcommand{\mcE}{\mathcal{E}}
\newcommand{\mcM}{\mathcal{M}}
\newcommand{\N}{\mathbb{N}}

\allowdisplaybreaks[1]

\begin{document}

\title[Toeplitzness of composition operators]{Toeplitzness of composition operators\\ in several variables}

\author{{\v Z}eljko {\v C}u{\v c}kovi{\'c}}
\address{Department of Mathematics and Statistics, Mail Stop 942, University of Toledo, Toledo, OH 43606}
\email{zeljko.cuckovic@utoledo.edu}

\author{Trieu Le}
\address{Department of Mathematics and Statistics, Mail Stop 942, University of Toledo, Toledo, OH 43606}
\email{trieu.le2@utoledo.edu}



\begin{abstract} Motivated by the work of Nazarov and Shapiro on the unit disk, we study asymptotic Toeplitzness of composition operators on the Hardy space of the unit sphere in $\C^n$. We extend some of their results but we also show that new phenomena appear in higher dimensions.
\end{abstract}

\maketitle

\section{Introduction}\label{S:intro}

Let $\B_n$ denote the unit ball and $\S_n$ the unit sphere in $\C^n$. We denote by $\sigma$ the  surface area measure on $\S_n$, so normalized that $\sigma(\S_n)=1$. We write $L^{\infty}$ for $L^{\infty}(\S_n,d\sigma)$ and $L^2$ for $L^2(\S_n,d\sigma)$. The Hardy space $H^2$ consists of all analytic functions $h$ on $\B_n$ which satisfy
\begin{equation*}
\|h\|^2 = \sup_{0<r<1}\int_{\S_n}|h(r\zeta)|^2\, d\sigma(\zeta) < \infty.
\end{equation*}
It is well known that such a function $h$ has radial boundary limits almost everywhere. We shall still denote the limiting function by $h$. We then have $h(\zeta) = \lim_{r\uparrow 1}h(r\zeta)$ for a.e. $\zeta\in\S$ and $$\|h\|^2 = \int_{\S_n}|h(\zeta)|^2\, d\sigma(\zeta) = \|h\|^2_{L^2}.$$
From this we may consider $H^2$ as a closed subspace of $L^2$. We shall denote by $P$ the orthogonal projection from $L^2$ onto $H^2$. We refer the reader to  \cite[Section 5.6]{RudinSpringer1980} for more details about $H^2$ and other Hardy spaces.

We shall also need the space $H^{\infty}$, which consists of bounded analytic functions on $\B_n$. As before, we may regard $H^{\infty}$ as a closed subspace of $L^{\infty}$.

For any $f\in L^{\infty}$, the Toeplitz operator $T_{f}$ is defined by $T_{f}h = P(fh)$ for $h$ in $H^2$. It is immediate that $T_{f}$ is bounded on $H^2$ with $\|T_{f}\|\leq\|f\|_{\infty}$. (The equality in fact holds true but it is highly nontrivial. We refer the reader to \cite{DavieJFA1977} for more details.) We call $f$ the symbol of $T_{f}$. The following properties are well known and can be verified easily from the definition of Toeplitz operators.
\begin{enumerate}[(a)]
  \item $T_{f}^{*} = T_{\overline{f}}$ for any $f\in L^{\infty}$.
  \item $T_{f}=M_f$, the multiplication operator with symbol $f$, for any $f\in H^{\infty}$. 
  \item $T_{g}T_{f}=T_{gf}$ and $T_{f}^{*}T_{g} = T_{\overline{f}g}$ for $f\in H^{\infty}$ and $g\in L^{\infty}$.
\end{enumerate}

The other class of operators that we are concerned with in this paper is the class of composition operators. Let $\varphi$ be an analytic mapping from $\B_n$ into itself. We shall call $\varphi$ an analytic selfmap of $\B_n$. We define the composition operator $C_{\varphi}$ by $C_{\varphi}h = h\circ\varphi$ for all analytic functions $h$ on $\B_n$. Note that $C_{\varphi}$ is the identity if and only if $\varphi$ is the identity mapping of $\B_n$. In the one dimensional case, it follows from Littlewood Subordination Principle that $C_{\varphi}$ is a bounded operator on the Hardy space $H^2$. In higher dimensions, $C_{\varphi}$ may not be bounded on $H^2$ even when $\varphi$ is a polynomial mapping. We refer the reader to \cite{CowenCRCP1995,Shapiro1993} for details on composition operators.

\subsection{One dimension}\label{SS:onedim}
In this section we discuss the case of one dimension, that is, $n=1$. It is a well known theorem of Brown and Halmos \cite{Brown1963} back in the sixties that a bounded operator $T$ on $H^2$ is a Toeplitz operator if and only if
\begin{equation}\label{Eqn:ToeplitzChar}
T_{\overline{z}}TT_{z}=T.
\end{equation}
Here $T_{z}$ is the Toeplitz operator with symbol $f(z)=z$ on the unit circle $\mathbb{T}$. This operator is also known as the unilateral forward shift. There is a rich literature on the study of Toeplitz operators and we refer the reader to, for example, \cite{Douglas1972} for more details. 

In their study of the Toeplitz algebra, Barr{\'i}a and Halmos \cite{BarriaTAMS1982} introduced the notion of asymptotic Toeplitz operators. An operator $A$ on $H^2$ is said to be \emph{strongly asymtotically Toeplitz} (``SAT'') if the sequence $\{T_{\overline{z}}^mAT_{z}^m\}_{m=0}^{\infty}$ converges in the strong operator topology. It is easy to verify, thanks to \eqref{Eqn:ToeplitzChar}, that the limit $A_{\infty}$, if exists, is a Toeplitz operator. The symbol of $A_{\infty}$ is called the asymptotic symbol of $A$. Barr{\'i}a and Halmos showed that any operator in the Toeplitz algebra is SAT. 

In \cite{FeintuchOTAA1989}, Feintuch investigated asymptotic Toeplitzness in the uniform (norm) and weak topology as well. An operator $A$ on $H^2$ is \emph{uniformly asymptotically Toeplitz} (``UAT'') (respectively, \emph{weakly asymptotically Toeplitz} (``WAT'')) if the sequence $\{T_{\overline{z}}^mAT_{z}^m\}$ converges in the norm (respectively, weak) topology. It is clear that $${\rm UAT} \Longrightarrow {\rm SAT} \Longrightarrow {\rm WAT}$$ and the limiting operators, if exist, are the same.

The following theorem of Feintuch completely characterizes operators that are UAT. A proof can be found in \cite{FeintuchOTAA1989} or \cite{NazarovCVEE2007}.
\begin{theorem}[Theorem 4.1 in \cite{FeintuchOTAA1989}]\label{T:Feintuch}
An operator on $H^2$ is uniformly asymptotically Toeplitz if and only if it has the form ``Toeplitz + compact''.
\end{theorem}

Recently Nazarov and Shapiro \cite{NazarovCVEE2007} investigated the asymptotic Toeplitzness of composition operators and their adjoints. They obtained many interesting results and open problems. We list here a few of their results, which are relevant to our work. 

\begin{theorem}[Theorem 1.1 in \cite{NazarovCVEE2007}]\label{T:NS}
$C_{\varphi}$ = ``Toeplitz + compact'' (or equivalently by Feintuch's Theorem, $C_{\varphi}$ is UAT) if and only if $C_{\varphi}=I$ or $C_{\varphi}$ is compact.
\end{theorem}

It is easy \cite[page 7]{NazarovCVEE2007} to see that if $\omega\in\partial\mathbb{D}\backslash\{1\}$ and $\varphi(z)=\omega z$ (such a $\varphi$ is called a rotation), then $C_{\varphi}$ is not WAT. On the other hand, Nazarov and Shapiro showed that for several classes of symbols $\varphi$, the operator $C_{\varphi}$ is WAT and the limiting operator is always zero. The following conjecture appeared in \cite{NazarovCVEE2007}.

\begin{WATconjecture} If $\varphi$ is neither a rotation nor the identity map, then $C_{\varphi}$ is WAT with asymptotic symbol zero.
\end{WATconjecture}

We already know that the conjecture holds when $C_{\varphi}$ is a compact operator. Nazarov and Shapiro showed that the conjecture also holds when (a) $\varphi(0)=0$; or (b) $|\varphi|=1$ on an open subset $V$ of $\mathbb{T}$ and $|\varphi|<1$ a.e. on $\mathbb{T}\backslash V$.

For the strong asymptotic Toeplitzness of composition operators, Nazarov and Shapiro proved several positive results. On the other hand, they showed that if $\varphi$ is a non-trivial automorphism of the unit disk, then $C_{\varphi}$ is not SAT. Later, {\v C}u{\v c}kovi{\'c} and Nikpour \cite{Cuckovic2010} proved that $C_{\varphi}^{*}$ is not SAT either. We combine these results into the following theorem.

\begin{theorem}\label{T:ZSCN}
Suppose $\varphi$ is a non-identity automorphism of $\D$. Then $C_{\varphi}$ and $C_{\varphi}^{*}$ are not SAT.
\end{theorem}

A more general notion of asymptotic Toeplitzness has been investigated by Matache in \cite{MatacheCHA2007}. An operator $S$ on $H^2$ is called a (generalized) unilateral forward shift if $S$ is an isometry and the sequence $\{S^{*m}\}$ converges to zero in the strong operator topology. An operator $A$ is called uniformly (strongly or weakly) $S$-asymptotically Toeplitz if the sequence $\{S^{*m}AS^m\}$ has a limit in the norm (strong or weak) topology. Among other things, the results in \cite{MatacheCHA2007} on the $S$-asymptotic Toeplitzness of composition operators generalize certain results in \cite{NazarovCVEE2007}.

\subsection{Higher dimensions}\label{SS:higherDimensions}

Motived by Nazarov and Shapiro's work discussed in the previous section, we would like to study the asymptotic Toeplitzness of composition operators on the Hardy space $H^2$ over the unit sphere in higher dimensions.  

To define the notion of asymptotic Toeplitzness, we need a characterization of Toeplitz operators. Such a characterization, which generalizes \ref{Eqn:ToeplitzChar}, was found by Davie and Jewell \cite{DavieJFA1977} back in the seventies. They showed that a bounded operator $T$ on $H^2$ is a Toeplitz operator if and only if $T = \sum_{j=1}^{n}T_{\overline{z}_j}TT_{z_j}$.

We define a linear operator $\Phi$ on the algebra $\mathcal{B}(H^2)$ of all bounded linear operators on $H^2$ by
\begin{equation}
\Phi(A) = \sum_{j=1}^{n}T_{\overline{z}_j}AT_{z_j},
\end{equation}
for any $A$ in $\mathcal{B}(H^2)$. It is clear that $\Phi$ is a positive map (that is, $\Phi(A)\geq 0$ whenever $A\geq 0$) and $\Phi$ is continuous in the weak operator topology of $\mathcal{B}(H^2)$. Let $S$ be the column operator whose components are $T_{z_1}, \ldots, T_{z_n}$. Then $S$ maps $H^2$ into the direct sum $(H^2)^{n}$ of $n$ copies of $H^2$.  In dimension $n=1$, the operator $S$ is the familiar forward unilateral shift. The adjoint $S^{*}=[T_{\overline{z}_1} \ldots T_{\overline{z}_n}]$ is a row operator from $(H^2)^n$ into $H^2$. Since $$S^{*}S = T_{\overline{z}_1}T_{z_1}+\cdots +T_{\overline{z}_n}T_{z_n} = T_{\overline{z}_1z_1+\cdots+\overline{z}_nz_n}=I,$$
we see that $S$ is a co-isometry. In particular, we have $\|S\|=\|S^{*}\|=1$.

From the definition of $\Phi$, we may write
\begin{equation*}
\Phi(A) = [T_{\overline{z}_1} \ldots T_{\overline{z}_n}]
\begin{bmatrix}
A & 0 & \dots & 0\\
0 & A & \dots & 0\\
\hdotsfor[2]{4}\\
0 & 0 & \dots & A
\end{bmatrix}
\begin{bmatrix}
T_{z_1}\\
\vdots\\
\vdots\\
T_{z_n}
\end{bmatrix}
= S^{*}\begin{bmatrix}
A & 0 & \dots & 0\\
0 & A & \dots & 0\\
\hdotsfor[2]{4}\\
0 & 0 & \dots & A
\end{bmatrix}
S.
\end{equation*}
It follows that $\|\Phi(A)\|\leq\|S^{*}\|\|A\|\|S\|\leq\|A\|$ for any $A$ in $\mathcal{B}(H^2)$. Hence $\Phi$ is a contraction. For any positive integer $m$, put $\Phi^m = \Phi\circ\cdots\circ\Phi$, the composition of $m$ copies of $\Phi$. Then we also have $\|\Phi^m(A)\|\leq\|A\|$.

The aforementioned Davie--Jewell's result shows that a bounded operator $T$ is a Toeplitz operator on $H^2$ if and only if $T$ is a fixed point of $\Phi$, which implies that $\Phi^m(T)=T$ for all positive integers $m$. 

We now define the notion of asymptotic Toeplitzness. An operator $A$ on $H^2$ is \textit{uniformly asymptotically Toeplitz} (``UAT'') (respectively, \textit{strongly asymptotically Toeplitz} (``SAT'') or \textit{weakly asymptotically Toeplitz} (``WAT'')) if the sequence $\{\Phi^m(A)\}$ converges in the norm topology (respectively, strong operator topology or weak operator topology). As in the one dimensional case, it is clear that $${\rm UAT} \Longrightarrow {\rm SAT} \Longrightarrow {\rm WAT}$$ and the limiting operators, if exist, are the same. Let $A_{\infty}$ denote the limiting operator. It follows from the continuity of $\Phi$ in the weak operator toplogy that $\Phi(A_{\infty})=A_{\infty}$. Therefore, $A_{\infty}$ is a Toeplitz operator. Write $A_{\infty}=T_{g}$ for some bounded function on $\S_n$. We shall call $g$ the asymptotic symbol of $A$.

{
In the definition of the map $\Phi$ (and hence the notion of Toeplitzness), we made use of the coordinate functions $z_1,\ldots,z_n$. It turns out that a unitary change of variables gives rise to the same map. More specifically, if $\{u_1,\ldots, u_n\}$ is any orthonormal basis of $\mathbb{C}^n$ and we define $f_j(z)=\langle z,u_j\rangle$ for $j=1,\ldots, n$ then a direct calculation shows that
\begin{align*}
\Phi(A) & = \sum_{j=1}^{n}T_{\overline{f}_j}AT_{f_j}
\end{align*}
for every bounded linear operator $A$ on $H^2$.
}

The rest of the paper is devoted to the study of the Toeplitzness of composition operators in several variables. Our focus is on strong and uniform asymptotic Toeplitzness. It turns out that while some results are analogous to the one dimensional case, other results are quite different.

\section{Strong asymptotic Toeplitzness}\label{S:SAT}
Let $\varphi = (\varphi_1,\ldots,\varphi_n)$ and $\eta = (\eta_1,\ldots,\eta_n)$ be two analytic selfmaps of $\B_n$. We also use $\varphi$ and $\eta$ to denote their radial limits at the boundary. For the rest of the paper, we will assume that both composition operators $C_{\varphi}$ and $C_{\eta}$ are bounded on the Hardy space $H^2$. (Recall that in dimensions greater than one, composition operators may not be bounded. See \cite[Section 3.5]{CowenCRCP1995} for more details.) Suppose $g$ is a bounded measurable function on $\S_n$. Using the identities $C_{\varphi}T_{z_j} = T_{\varphi_j}C_{\varphi}$ and $T_{\overline{z}_j}C^{*}_{\eta} = C^{*}_{\eta}T_{\overline{\eta}_j}$ for $j=1,\ldots, n$, we obtain
\begin{align*}
\Phi(C_{\eta}^{*}T_{g}C_{\varphi}) & = \sum_{j=1}^{n}T_{\overline{z}_j}C_{\eta}^{*}T_{g}C_{\varphi}T_{z_j} = \sum_{j=1}^{n}C_{\eta}^{*}T_{\overline{\eta}_j}T_{g}T_{\varphi_j}C_{\varphi}\\
& = C_{\eta}^{*}\Big(\sum_{j=1}^{n}T_{\overline{\eta}_jg\varphi_j}\Big)C_{\varphi} = C_{\eta}^{*}T_{g\langle\varphi,\eta\rangle}C_{\varphi}.
\end{align*}
Here $\langle\varphi,\eta\rangle$ is the inner product of $\varphi=\langle\varphi_1,\ldots,\varphi_n\rangle$ and $\eta=\langle\eta_1,\ldots,\eta_n\rangle$ as vectors in $\C^n$.
By induction, we conclude that
\begin{equation}\label{Eqn:formulaPhi}
\Phi^m(C_{\eta}^{*}T_{g}C_{\varphi}) = C_{\eta}^{*}T_{g\langle\varphi,\eta\rangle^m}\ C_{\varphi} \quad\text{  for any } m\geq 1.
\end{equation}

As an immediate application of the formula \ref{Eqn:formulaPhi}, we show that certain products of Toeplitz and composition operators on $H^2$ are SAT.

\begin{proposition}\label{P:SAT_zero_ae} Suppose that $|\langle\varphi,\eta\rangle|<1$ a.e. on $\S_n$. Then for any bounded function $g$ on $\S_n$, the operator $C^{*}_{\eta}T_{g}C_{\varphi}$ is SAT with asymptotic symbol zero.
\end{proposition}

\begin{proof}
By assumption, $\langle\varphi,\eta\rangle^{m}\rightarrow 0$ a.e. on $\S_{n}$ as $m\to\infty$. This, together with Lebesgue Dominated Convergence Theorem, implies that $T_{g\langle\varphi,\eta\rangle^{m}}\rightarrow 0$, and hence, $C_{\eta}^{*}T_{g\langle\varphi,\eta\rangle^m}\ C_{\varphi}\rightarrow 0$ in the strong operator topology. Using \eqref{Eqn:formulaPhi}, we conclude that $\Phi^{m}(C^{*}_{\eta}T_{g}C_{\varphi})\rightarrow 0$ in the strong operator topology. The conclusion of the proposition follows.
\end{proof}

As suggested by \ref{Eqn:formulaPhi}, the following set is relevant to the study of the asmytotic Toeplitzness of $C^{*}_{\eta}T_{g}C_{\varphi}$:
\begin{align*}
E(\varphi,\eta) & = \big\{\zeta\in\S_n: \langle\varphi(\zeta),\eta(\zeta)\rangle = 1\big\}\\
 & = \big\{\zeta\in\S_n: \varphi(\zeta)=\eta(\zeta) \text{ and } |\varphi(\zeta)|=1\big\}.
\end{align*}
To obtain the second equality we have used the fact that $|\varphi(\zeta)|\leq 1$ and $|\eta(\zeta)|\leq 1$ for $\zeta\in\S_n$. Note that $E(\varphi,\varphi)$ is the set of all $\zeta\in\S_n$ for which $|\varphi(\zeta)|=1$. On the other hand, by \cite[Theorem 5.5.9]{RudinSpringer1980}, if $\varphi\neq\eta$, then $E(\varphi,\eta)$ has measure zero. 

\begin{proposition}\label{P:MSAT} For any analytic selfmaps $\varphi,\eta$ of $\B_n$ and any bounded function $g$ on $\S_n$, we have
\begin{align*}
\frac{1}{m}\sum_{j=1}^{m}\Phi^{j}(C_{\eta}^{*}T_{g}C_{\varphi}) \longrightarrow C^{*}_{\eta}T_{g\chi_{E(\varphi,\eta)}}C_{\varphi} \text{ in the strong operator topology}
\end{align*}
as $m\to\infty$.
\end{proposition}

\begin{proof}
By \eqref{Eqn:formulaPhi}, it suffices to show that $(1/m)\sum_{j=1}^m g\langle\varphi,\eta\rangle^j$ converges to $g\chi_{E(\varphi,\eta)}$ a.e. on $\S_n$. But this follows from the identity
\begin{align*}
\frac{1}{m}\sum_{j=1}^m g(\zeta)\langle\varphi(\zeta),\eta(\zeta)\rangle^j & = \begin{cases}
g(\zeta) & \quad\text{ if } \zeta\in E(\varphi,\eta)\\
\frac{1}{m}g(\zeta)\Big(\frac{1-\langle\varphi(\zeta),\eta(\zeta)\rangle^{m+1}}{1-\langle\varphi(\zeta),\eta(\zeta)\rangle}\Big) & \quad\text{ if } \zeta\notin E(\varphi,\eta),
\end{cases}
\end{align*}
for any $\zeta\in\S_n$.
\end{proof}

Proposition \ref{P:MSAT} says that any operator of the form $C^{*}_{\eta}T_gC_{\varphi}$ is \textit{mean strongly asymptotically Toeplitz} (``MSAT'') with limit $C^{*}_{\eta}T_{g\chi_{E(\varphi,\eta)}}C_{\varphi}$. We now specify $\eta$ to be the identity map of $\B_n$ and $g$ to be the constant function $1$ and obtain

\begin{corollary}\label{C:MSAT} Let $\varphi$ be a non-identity analytic selfmap of $\B_n$ such that $C_{\varphi}$ is bounded on $H^2$. Then $C_{\varphi}$ is MSAT with asymptotic symbol zero. 
\end{corollary}

This result in the one-dimensional case was obtained by Shapiro in \cite{ShapiroJMAA2007}. In fact, Shapiro considered a more general notion of MSAT. It seems possible to generalize Proposition \ref{P:MSAT} in that direction and we leave this for the interested reader.

Theorem \ref{T:ZSCN} asserts that for $\varphi$ a non-identity automorphism of the unit disk $\D$, the operators $C_{\varphi}$ and $C_{\varphi}^{*}$ are not SAT. In dimensions greater than one, the situation is different. To state our result, we first fix some notation. Let $\mathcal{A}(\B_n)$ denote the space of functions that are analytic on the open unit ball $\B_n$ and continuous on the closure $\overline{\B}_n$. We also let ${\rm Lip}(\alpha)$ (for $0<\alpha\leq 1$) be the space of $\alpha$-Lipschitz continuous functions on $\B_n$, that is, the space of all functions $f:\B_n\rightarrow\C$ such that
\begin{align*}
\sup\Big\{\dfrac{|f(a)-f(b)|}{|a-b|^{\alpha}}: a,b\in\B_n, a\neq b\Big\} < \infty.
\end{align*}
We shall need the following result, see \cite[p.248]{RudinSpringer1980}.
\begin{proposition}\label{P:RuSi}
Suppose $n\geq 2$. If $1/2<\alpha\leq 1$ and $f\in\mathcal{A}(\B_n)\cap {\rm Lip}(\alpha)$ is not a constant function, then
\begin{align*}
\sigma\Big(\big\{\zeta\in\S_n: |f(\zeta)|=\|f\|_{\infty}\big\}\Big) = 0.
\end{align*}
\end{proposition}

\begin{theorem}\label{T:SAT_affine} Suppose $n\geq 2$. Let $A:\C^n\rightarrow\C^n$ be a linear operator and $b$ be a vector in $\C^n$. Let $f$ be in $\mathcal{A}(\B_n)\cap {\rm Lip}(\alpha)$ for some $1/2<\alpha\leq 1$. Suppose $\varphi(z)=f(z)(Az+b)$ is a selfmap of $\B_n$ and $\varphi$ is not of the form $\varphi(z)=\lambda z$ with $|\lambda|=1$. Then both $C_{\varphi}$ and $C_{\varphi}^{*}$ are SAT with asymptotic symbol zero.
\end{theorem}

{
Before giving a proof of the theorem, we present here an immediate application. For any $n\geq 1$, a linear fractional mapping of the unit ball $\B_{n}$ has the form $$\varphi(z)=\frac{Az+B}{\langle z,C\rangle+D},$$ where $A$ is a linear map, $B, C$ are vectors in $\C^n$ and $D$ is a non-zero complex number. It was shown by Cowen and MacCluer that $C_{\varphi}$ is always bounded on $H^2$ for any linear fractional selfmap $\varphi$ of $\B_n$. We recall that when $n=1$ these operators and their adjoints are not SAT in general by Theorem \ref{T:ZSCN}. In higher dimensions it follows from Theorem \ref{T:SAT_affine} that the opposite is true.

\begin{corollary}
\label{C:SAT_LFT}
For $n\geq 2$, both $C_{\varphi}$ and $C_{\varphi}^{*}$ are SAT with asymptotic symbol zero except in the case $\varphi(z)=\lambda z$ for some $\lambda\in\mathbb{T}$. 
\end{corollary}
}

\begin{proof}[Proof of Theorem \ref{T:SAT_affine}] We claim that under the hypothesis of the theorem, the set $$\mcE=\big\{\zeta\in\S_n: |\langle\varphi(\zeta),\zeta\rangle|=1\big\}$$
is a $\sigma-$null subset of $\S_n$. We may then apply Proposition \ref{P:SAT_zero_ae}. 

There are two cases to consider.

\textbf{Case 1.} $A=\delta I$ for some complex number $\delta$ and $b=0$. To simplify the notation, we write $\varphi(z)=g(z)z$, where $g(z)=\delta f(z)$. Then the set $\mcE$ can be written as $\mcE = \{\zeta\in\S_n: |g(\zeta)|=1\}$.

Since $\varphi$ is a selfmap of $\B_n$, we have $\|g\|_{\infty}\leq 1$. Now if $\|g\|_{\infty}<1$, then $\mcE=\emptyset$ so $\sigma(\mcE)=0$. If $\|g\|_{\infty}=1$, then $g$ is a non-constant function since $\varphi$ is not of the form $\varphi(z)=\lambda z$ for some $|\lambda|=1$. Proposition \ref{P:RuSi} then gives $\sigma(\mcE)=0$ as well.

\textbf{Case 2.} $A$ is not a multiple of the identity or $b\neq 0$. Since $|\varphi(\zeta)|\leq 1$ for $\zeta\in\S_n$, we see that $\zeta$ belongs to $\mcE$ if and only if there is a unimodular complex number $\gamma(\zeta)$ such that $\varphi(\zeta)=\gamma(\zeta)\zeta$. This implies that $f(\zeta)\neq 0$ and 
\begin{align}\label{Eqn:setE}
\big(A-\gamma(\zeta)/f(\zeta)\big)\zeta+b=0.
\end{align}
Equation \eqref{Eqn:setE} shows that $\mcE$ is contained in the intersection of $\S_n$ with the set
\begin{align*}
\mcM=\Big\{z\in\C^n: (A-\lambda)z+b=0\,\text{ for some } \lambda\in\C\Big\} = \bigcup_{\lambda\in\mathbb{C}}(A-\lambda)^{-1}(\{-b\}).
\end{align*}
Now decompose $\mcM$ as the union $\mcM=\mcM_1 \cup \mcM_2$, where
\begin{align*}
\mcM_1 = \bigcup_{\lambda\in\C\backslash{\rm sp}(A)}(A-\lambda)^{-1}(\{-b\}) \quad\text{ and } \quad \mcM_2 =  \bigcup_{\lambda\in{\rm sp}(A)}(A-\lambda)^{-1}(\{-b\}).
\end{align*}
We have used ${\rm sp}(A)$ to denote the spectrum of $A$, which is just the set of eigenvalues since $A$ is an operator on $\C^n$. We shall show that both sets $\mcM_1\cap\S_n$ and $\mcM_2\cap\S_n$ are $\sigma$-null sets.

For $\lambda\in\C\backslash{\rm sp}(A)$, the equation $(A-\lambda)z+b=0$ has a unique solution whose components are rational functions in $\lambda$ by Cramer's rule. So $\mcM_1$ is a rational curve parametrized by $\lambda\in\C\backslash{\rm sp}(A)$. Since the real dimension of $\S_n$ is $2n-1$, which is at least $3$ when $n\geq 2$, we conclude that $\sigma(\mcM_1\cap\S_n)=0$. 

For $\lambda\in{\rm sp}(A)$, the set $(A-\lambda)^{-1}(\{-b\})$ is either empty or an affine subspace of complex dimension at most $n-1$ (hence, real dimension at most $2n-2$). Since $\mcM_2$ is a union of finitely many such sets and the sphere $\S_n$ has real dimension $2n-1$, we conclude that $\sigma(\mcM_2\cap\S_n)=0$.

Since $\mcE\subset (\mcM_1\cup\mcM_2)\cap\S_n$ and $\sigma(\mcM_1\cap\S_n)=\sigma(\mcM_2\cap\S_n)=0$, we have $\sigma(\mcE)=0$, which completes the proof of the claim.
\end{proof}

Nazarov and Shapiro \cite{NazarovCVEE2007} showed in the one-dimensional case that if $\varphi$ is an inner function which is not of the form $\lambda z$ for some constant $\lambda$, and $\varphi(0)=0$, then $C_{\varphi}$ is not SAT but $C_{\varphi}^{*}$ is SAT. While we do not know what the general situation is in higher dimensions, we have obtained a partial result. 

\begin{proposition}
Suppose $f$ is a non-constant inner function on $\B_{n}$ and $\varphi(z)=f(z)z$ for $z\in\B_{n}$ such that $C_{\varphi}$ is bounded on $H^2$. Then $C_{\varphi}$ is not SAT but $C_{\varphi}^{*}$ is SAT.
\end{proposition}

\begin{proof}
By formula \eqref{Eqn:formulaPhi}, we have $\Phi^{m}(C_{\varphi})=T_{f^{m}}C_{\varphi}$ and $\Phi^{m}(C^{*}_{\varphi})=C_{\varphi}^{*}T^{*}_{{f}^{m}}$ for all positive integers $m$.

It then follows that $\|\Phi^m(C_{\varphi})(1)\| = \|T_{f^m}C_{\varphi}1\| = \|f^m\|=1$. Hence $\Phi^m(C_{\varphi})$ does not converge to zero in the strong operator topology. Since $\varphi$ is a non-identity selfmap of $\B_n$, Corollary \ref{C:MSAT} implies that $C_{\varphi}$ is not SAT.

On the other hand, we claim that as $m\to\infty$, $T^{*}_{{f}^m}$, and hence, $\Phi^m(C_{\varphi}^{*})$, converges to zero in the strong operator topology. This shows that $C_{\varphi}^{*}$ is SAT with asymptotic symbol zero. The proof of the claim is similar to that in case of dimension one (\cite[Theorem 4.2]{NazarovCVEE2007}). For the reader's convenience, we provide here the details. For any $a\in\B_n$, there is a function $K_{a}\in H^2$ such that $h(a)=\langle h,K_a\rangle$ for any $h\in H^2$. Such a function is called a reproducing kernel. It is well known that $T^{*}_{{f}^m}K_{a} = \overline{f^m(a)}K_{a}$ for any integer $m\geq 1$. Since $|f(a)|<\|f\|_{\infty}=1$ by the Maximum Principle, it follows that $\|T^{*}_{{f}^m}K_{a}\|\rightarrow 0$ as $m\to\infty$. Because the linear span of $\{K_a: a\in\B_n\}$ is dense in $H^2$ and the operator norms of $\|T^{*}_{{f}^m}\|$ are uniformly bounded by one, we conclude that $T^{*}_{{f}^m}\rightarrow 0$ in the strong operator topology.
\end{proof}

\section{Uniform asymptotic Toeplitzness}

It follows from the characterization of Toeplitz operators and the notion of Toeplitzness that any Toeplitz operator is UAT. The following lemma shows that any compact operator is also UAT. Hence, anything of the form ``Toeplitz + compact'' is UAT. This result may have appeared in the literature but for completeness, we sketch here a proof.

\begin{lemma} \label{L:limitCompact}
Let $K$ be a compact operator on $H^2$. Then we have $$\lim_{m\to\infty}\|\Phi^{m}(K)\| = 0.$$ As a consequence, for any bounded function $f$, the operator $T_{f}+K$ is uniformly asymptotically Toeplitz with asymptotic symbol $f$.
\end{lemma}

\begin{proof}
Since $\Phi^m$ is a contraction for each $m$ and any compact operator can be approximated in norm by finite-rank operators, it suffices to consider the case when $K$ is a rank-one operator. Write $K=u\otimes v$ for some non-zero vectors $u,v\in H^2$. Here $(u\otimes v)(h) = \langle h,v\rangle u$ for $h\in H^2$. Since polynomials form a dense set in $H^2$, we may assume further that both $u,v$ are polynomials. 

For any multi-index $\alpha$, we have $
T_{\overline{z}^{\alpha}}\big(u\otimes v\big)T_{z^{\alpha}} = \big(T_{\overline{z}^{\alpha}}u \big)\otimes\big(T_{\overline{z}^{\alpha}}v\big).
$
Since $v$ is a polynomial, there exists an integer $m_0$ such that $T_{\overline{z}^{\alpha}}v = 0$ for any  $\alpha$ with $|\alpha|>m_0$. If $m$ is a positive integer, the definition of $\Phi$ shows that $\Phi^{m}(K)=\Phi^{m}(u\otimes v)$ is a finite sum of operators of the form $T_{\overline{z}^{\alpha}}\big(u\otimes v\big)T_{z^{\alpha}}$ with $|\alpha|=m$. This implies that  $\Phi^{m}(K)=0$ for all $m>m_0$. Therefore, $\lim_{m\to\infty}\|\Phi^{m}(K)\|=0$.

Now for $f$ a bounded function on $\S_n$, we have
\begin{align*}
\Phi^m(T_{f}+K) & = \Phi^m(T_{f})+\Phi^m(K) = T_{f} + \Phi^m(K) \longrightarrow T_{f}
\end{align*}
in the norm topology as $m\to\infty$. This shows that $T_{f}+K$ is UAT with asymptotic symbol $f$.
\end{proof}

In dimension one, Theorem \ref{T:Feintuch} shows that the converse of Lemma \ref{L:limitCompact} holds. On the other hand, Theorem \ref{T:Feintuch} fails when $n\geq 2$. We shall show that there exist composition operators that are UAT but cannot be written in the form ``Toeplitz + compact''.

We first show that composition operators cannot be written in the form ``Toeplitz + compact'' except in trivial cases. This generalizes Theorem \ref{T:NS} to all dimensions.
\begin{theorem}\label{T:cptPlusToeplitz}
Let $\varphi$ be an analytic selfmap of $\B_n$ such that $C_{\varphi}$ is bounded on $H^2$. If $C_{\varphi}$ can be written in the form ``Toeplitz + compact'', then either $C_{\varphi}$ is compact or it is the identity operator.
\end{theorem}

\begin{proof}
Our proof here works also for the one-dimensional case and it is different from Nazarov--Shapiro's approach (see the proof of Theorem 1.1 in \cite{NazarovCVEE2007}). Suppose $C_{\varphi}$ is not the identity and $C_{\varphi} = T_f+K$ for some compact operator $K$ and some bounded function $f$. By Lemma \ref{L:limitCompact}, $C_{\varphi}$ is UAT with asymptotic symbol $f$ on the unit sphere. This then implies that $C_{\varphi}$ is also MSAT with asymptotic symbol $f$. From Corollary \ref{C:MSAT} we know that $C_{\varphi}$, being a non-identity bounded composition operator, is MSAT with asymptotic symbol zero. Therefore $f=0$ a.e. and hence $C_{\varphi}=K$. This completes the proof of the theorem.
\end{proof}

We now provide an example which shows that the converse of Lemma \ref{L:limitCompact} (and hence Theorem \ref{T:Feintuch}) does not hold in higher dimensions.

\begin{example}\label{E:counterexample}
For $z=(z_1,\ldots,z_n)$ in $\B_n$, we define $$\varphi(z)=(\varphi_1(z),\ldots,\varphi_n(z))=(0,z_1,0,\ldots,0).$$ Then $\varphi$ is a linear operator that maps $\B_n$ into itself.  It follows from  \cite[Lemma 8.1]{CowenCRCP1995} that $C_{\varphi}$ is bounded on $H^2$ and $C_{\varphi}^{*} = C_{\psi}$, where $\psi$ is a linear map given by $\psi(z)=(\psi_1(z),\ldots,\psi_n(z))=(z_2,0,\ldots,0)$.

We claim that $\Phi(C_{\varphi})=0$. For $j\neq 2$, $C_{\varphi}T_{z_j}=T_{\varphi_j}C_{\varphi}=0$ since $\varphi_j=0$ for such $j$. Also, $(T_{\overline{z}_2}C_{\varphi})^{*} = C_{\varphi}^{*}T_{z_2} = C_{\psi}T_{z_2} = T_{\psi_2}C_{\psi} = 0.$
Hence $T_{\overline{z}_2}C_{\varphi}=0$.  It follows that $\Phi(C_{\varphi})=T_{\overline{z}_1}C_{\varphi}T_{z_1}+\cdots+T_{\overline{z}_n}C_{\varphi}T_{z_n}=0$, which implies $\Phi^m(C_{\varphi})=0$ for all $m\geq 1$. Thus, $C_{\varphi}$ is UAT with asymptotic symbol zero. 

On the other hand, since $(\varphi\circ\psi)(z) = (0,z_2,0,\ldots 0)$, we conclude that for any non-negative integer $s$,
$$C^{*}_{\varphi}C_{\varphi}(z_2^{s}) = C_{\psi}C_{\varphi}(z_2^{s}) = C_{\varphi\circ\psi}(z_2^{s})=z_2^{s}.$$ This shows that the restriction of $C_{\varphi}$ on the infinite dimensional subspace spanned by $\{1,z_2,z_2^2,z_2^3,\ldots\}$ is an isometric operator. As a consequence, $C_{\varphi}$ is not compact on $H^2$. Theorem \ref{T:cptPlusToeplitz} now implies that $C_{\varphi}$ is not of the form ``Toeplitz + compact'' either.
\end{example}

Theorem \ref{T:NS} shows that on the Hardy space of the unit disk, a composition operator $C_{\varphi}$ is UAT if and only if it is either a compact operator or the identity. Example \ref{E:counterexample} shows that in dimensions $n\geq 2$, there exists a non-compact, non-identity composition operator which is UAT. It turns out that there are many more such composition operators. In the rest of the section, we study uniform asymptotic Toeplitzness of composition operators induced by linear selfmaps of $\B_n$.

We begin with a proposition which gives a lower bound for the norm of the product $T_{f}C_{\varphi}$ when $\varphi$ satisfies certain conditions. This estimate will later help us show that certain composition operators are not UAT.

\begin{proposition}
\label{P:normTC}
Let $\varphi$ be an analytic selfmap of $\B_n$ such that $C_{\varphi}$ is bounded. Suppose there are points $\zeta, \eta\in\S_n$ so that $\langle\varphi(z),\eta\rangle=\langle z,\zeta\rangle$ for a.e. $z\in\S_n$. Let $f$ be a bounded function on $\S_n$ which is continuous at $\zeta$. Then we have 
\begin{align*}
\|T_{f}C_{\varphi}\| \geq |f(\zeta)|.
\end{align*}
\end{proposition}

\begin{proof}
For an integer $s\geq 1$, put $g_{s}(z)=\big(1+\langle z,\eta\rangle\big)^{s}$ and $h_{s}=C_{\varphi}g_{s}$. Then for a.e. $z\in\S_n$,
\begin{align*}
h_s(z) & = g_s(\varphi(z)) = \big(1+\langle \varphi(z),\eta\rangle\big)^{s} = \big(1+\langle z,\zeta\rangle\big)^{s}.
\end{align*}
Because of the rotation-invariance of the surface measure on $\S_n$, we see that $\|h_s\|=\|g_s\|$. Now, we have
\begin{align*}
\|T_{f}C_{\varphi}\| \geq \dfrac{|\langle T_{f}C_{\varphi}g_s,h_s\rangle|}{\|g_s\|\|h_s\|} & = \dfrac{|\langle T_{f}h_s,h_s\rangle|}{\|g_s\|\|h_s\|} = \dfrac{|\langle fh_s,h_s\rangle|}{\|h_s\|^2}\\
& = \Big|\dfrac{\int_{\S_n}f(z)|1+\langle z,\zeta\rangle|^{2s}\, d\sigma(z)}{\int_{\S_n}|1+\langle z,\zeta\rangle|^{2s}\,d\sigma(z)}\Big|.
\end{align*}
We claim that the limit as $s\to\infty$ of the quantity inside the absolute value is $f(\zeta)$. From this the conclusion of the proposition follows.

To prove the claim we consider
\begin{align}
\label{Eqn:limit_estimate}
& \Big|\dfrac{\int_{\S_n}f(z)|1+\langle z,\zeta\rangle|^{2s}\, d\sigma(z)}{\int_{\S_n}|1+\langle z,\zeta\rangle|^{2s}\,d\sigma(z)} - f(\zeta)\Big| \leq \dfrac{\int_{\S_n}|f(z)-f(\zeta)|\cdot|1+\langle z,\zeta\rangle|^{2s}\, d\sigma(z)}{\int_{\S_n}|1+\langle z,\zeta\rangle|^{2s}\,d\sigma(z)}\notag\\
& \quad\quad = \dfrac{\big(\int_{\mathcal U}+\int_{\S_n\backslash\mathcal{U}}\big)|f(z)-f(\zeta)|\cdot|1+\langle z,\zeta\rangle|^{2s}\, d\sigma(z)}{\int_{\S_n}|1+\langle z,\zeta\rangle|^{2s}\,d\sigma(z)}\notag\\
& \quad\quad \leq\sup_{z\in\mathcal{U}}|f(z)-f(\zeta)|+2\|f\|_{\infty}\dfrac{\int_{\S_n\backslash\mathcal{U}}|1+\langle z,\zeta\rangle|^{2s}\, d\sigma(z)}{\int_{\S_n}|1+\langle z,\zeta\rangle|^{2s}\,d\sigma(z)},
\end{align}
where $\mathcal{U}$ is any open neighborhood of $\zeta$ in $\S_n$. By the continuity of $f$ at $\zeta$, the first term in \eqref{Eqn:limit_estimate} can be made arbitrarily small by choosing an appropriate $\mathcal{U}$. For such a $\mathcal{U}$, we may choose another open neighborhood $\mathcal{W}$ of $\zeta$ with $\mathcal{W}\subseteq\mathcal{U}$ such that
\begin{align*}
\sup\big\{|1+\langle z,\zeta\rangle|: \, z\in\S_{n}\backslash\mathcal{U}\big\} \,<\, \inf\big\{|1+\langle z,\zeta\rangle|: \, z\in\mathcal{W}\big\}.
\end{align*}
This shows that the second term in \eqref{Eqn:limit_estimate} converges to $0$ as $s\to\infty$. The claim then follows.
\end{proof}

Using Proposition \ref{P:normTC}, we give a sufficient condition under which $C_{\varphi}$ fails to be UAT.
\begin{proposition}
\label{P:nonUAT}
Let $\varphi$ be a non-identity analytic selfmap of $\B_n$ such that $C_{\varphi}$ is bounded. Suppose that $\varphi$ is continuous on $\overline{\B}_n$ and there is a point $\zeta\in\S_n$ and a unimodular complex number $\lambda$ so that $\langle\varphi(z),\zeta\rangle=\lambda\langle z,\zeta\rangle$ for all $z\in\S_n$. Then $C_{\varphi}$ is not UAT.
\end{proposition}

\begin{proof}
Since $\varphi$ is a non-identity map, Corollary \ref{C:MSAT} shows that $C_{\varphi}$ is MSAT with asymptotic symbol zero. To prove that $C_{\varphi}$ is not UAT, it suffices to show that $C_{\varphi}$ is not UAT with asymptotic symbol zero. 

Let $f(z)=\langle\varphi(z),z\rangle$ for $z\in\S_n$. By the hypothesis, the function $f$ is continuous on $\S_n$ and $f(\zeta)=\langle\varphi(\zeta),\zeta\rangle = \lambda\langle\zeta,\zeta\rangle=\lambda$. For any positive integer $m$, formula \eqref{Eqn:formulaPhi} gives $\Phi^m(C_{\varphi}) = T_{f^m}C_{\varphi}$. Since $\varphi$ satisfies the hypothesis of Proposition \ref{P:normTC} with $\eta=\lambda\zeta$ and $f^m$ is continuous at $\zeta$, we may apply Proposition \ref{P:normTC} to conclude that
\begin{align*}
\|\Phi^m(C_{\varphi})\| = \|T_{f^m}C_{\varphi}\| & \geq |f^m(\zeta)| = 1.
\end{align*}
This implies that $C_{\varphi}$ is not UAT with asymptotic symbol zero, which is what we wished to prove.
\end{proof}

Our last result in the paper provides necessary and sufficient conditions for a class of composition operators to be UAT.
\begin{theorem}
Let $\varphi(z)=Az$ where $A:\mathbb{C}^n\rightarrow\mathbb{C}^n$ is a non-identity linear map with $\|A\|\leq 1$. Then $C_{\varphi}$ is UAT if and only if all eigenvalues of $A$ lie inside the open unit disk.
\end{theorem}

\begin{proof}
Since $\|A\|\leq 1$, all eigenvalues of $A$ lie inside the closed unit disk. We first show that if $A$ has an eigenvalue $\lambda$ with $|\lambda|=1$, then $C_{\varphi}$ is not UAT. Let $\zeta\in\S_n$ be an eigenvector of $A$ corresponding to $\lambda$. We claim that $A^{*}\zeta = \overline{\lambda}\zeta$. In fact, we have
\begin{align*}
|(A^{*}-\overline{\lambda})\zeta|^2 & = |A^{*}\zeta|^2-2\Re\langle A^{*}\zeta,\overline{\lambda}\zeta\rangle+|\overline{\lambda}\zeta|^2 \\
& = |A^{*}\zeta|^2 - 2\Re\langle\zeta,\overline{\lambda}A\zeta\rangle+|\zeta|^2\\
& = |A^{*}\zeta|^2 - 2\Re\langle\zeta,\overline{\lambda}\lambda\zeta\rangle+|\zeta|^2\\
& =|A^{*}\zeta|^2-1\leq 0.
\end{align*}
This forces $A^{*}\zeta=\overline{\lambda}\zeta$ as claimed. As a result, for $z\in\S_n$, we have $$\langle\varphi(z),\zeta\rangle = \langle Az,\zeta\rangle = \langle z,A^{*}\zeta\rangle = \langle z,\overline{\lambda}\zeta\rangle = \lambda\langle z,\zeta\rangle.$$
We then apply Proposition \ref{P:nonUAT} to conclude that $C_{\varphi}$ is not UAT.

We now show that if all eigenvalues of $A$ lie inside the open unit disk then $C_{\varphi}$ is UAT. Put $f(z)=\langle\varphi(z),z\rangle = \langle Az,z\rangle$ for $z\in\S_n$. Since $|Az|\leq 1$ and $Az$ is not a unimodular multiple of $z$, we see that $|f(z)|<1$ for $z\in\S_n$. Since $f$ is continuous and $\S_n$ is compact, we have $\|f\|_{L^{\infty}(\S_n)}<1$. For any integer $m\geq 1$, formula \eqref{Eqn:formulaPhi} gives
\begin{align*}
\|\Phi^m(C_{\varphi})\| & = \|T_{f^m}C_{\varphi}\|\leq \|T_{f^m}\|\|C_{\varphi}\| \leq (\|f\|_{L^{\infty}(\S_n)})^{m}\|C_{\varphi}\|.
\end{align*}
Since $\|f\|_{L^{\infty}(\S_n)}<1$, we conclude that $\lim_{m\to\infty}\|\Phi^m(C_{\varphi})\|=0$. Therefore, $C_{\varphi}$ is UAT with asymptotic symbol zero.
\end{proof}

\subsection*{Acknowledgements} We wish to thank the referee for useful suggestions that improved the presentation of the paper.

\end{document}